\documentclass[a4paper, reqno]{amsart}
\usepackage[margin=3.5cm]{geometry}

\usepackage{latexsym}
\usepackage{amsmath}
\usepackage{epsf}
\usepackage[latin1]{inputenc}
\usepackage{graphics, a4wide}
\usepackage{enumerate}
\usepackage{calc}
\usepackage{amssymb}
\usepackage[active]{srcltx}
\usepackage{epsfig}
\usepackage{pdfsync}
\usepackage[mathscr]{euscript}

\usepackage{color}
\usepackage{pdfsync}

\usepackage{amsthm}
\usepackage{amscd}
\usepackage{esint}

\newcommand{\F}{\mathcal{F}}

\renewcommand{\a}{\alpha}

\newcommand{\D}{\nabla}

\def\d{\partial}

\def\e{\varepsilon}

\def\theta{\vartheta}

\def\weak{\rightharpoonup}
\newcommand{\R}{\mathbb{R}}

\def\FF{\mathcal{F}}

\def\D{\nabla}
\def\a{\alpha}

\def\XXint#1#2#3{{\setbox0=\hbox{$#1{#2#3}{\int}$}
     \vcenter{\hbox{$#2#3$}}\kern-.5\wd0}}

\newcommand{\We}{W^{2,2}_{{\rm det}, \e}(S)}

\newcounter{bei}

\newcommand{\zwo}[2]{\begin{pmatrix} {#1}\\{#2} \end{pmatrix}}

\renewcommand{\t}{\widetilde}
\renewcommand{\o}{\overline}

\newcommand{\De}{\nabla_{\!\e}}

\numberwithin{equation}{section}

\theoremstyle{plain}
\begingroup
\newtheorem{theorem}{Theorem}[section]
\newtheorem{lemma}[theorem]{Lemma}

\endgroup

\theoremstyle{definition}
\begingroup

\newtheorem{remark}[theorem]{Remark}
\endgroup

\DeclareMathOperator{\sym}{sym}

\author[L. Freddi]{Lorenzo Freddi}
\author[P. Hornung]{Peter Hornung}
\author[M.G. Mora]{Maria Giovanna Mora}
\author[R. Paroni]{Roberto Paroni}

\begin{document}
\title{One-dimensional von K\'arm\'an models for elastic ribbons}

\address[L. Freddi]{Dipartimento di Matematica e Informatica,
via delle Scienze 206, 33100 Udine, Italy
}\email{lorenzo.freddi@uniud.it}
\address[P. Hornung]{Fachrichtung Mathematik,
TU Dresden,
01062 Dresden,
Germany
}\email{peter.hornung@tu-dresden.de}
\address[M.G. Mora]{Dipartimento di Matematica, Universit\`a di Pavia, via Ferrata 1, 27100 Pavia, Italy}
\email{mariagiovanna.mora@unipv.it}
\address[R. Paroni]{DADU, Universit\`{a} degli Studi di
Sassari, Palazzo del Pou Salit,
07041 Alghero (SS), Italy} \email{paroni@uniss.it}

\date{}

\begin{abstract}
By means of a variational approach we rigorously deduce three one-dimensional models for elastic ribbons from the theory of von  K\'arm\'an plates, passing to the limit as the width of the plate goes to zero.
The one-dimensional model found starting from the ``linearized'' von  K\'arm\'an energy corresponds to that of a linearly elastic beam  that can twist but can deform in just one plane; while the model found from the  von  K\'arm\'an energy is a non-linear model that comprises stretching, bendings, and twisting.  The ``constrained'' von  K\'arm\'an energy, instead, leads to a new Sadowsky type of model.
\\

\medskip


\end{abstract}

\maketitle

\section{Introduction}\label{Sec0}

Geometrically a ribbon is a body with three length scales: it is a parallelepiped whose length $\ell$ is much larger than the width $\e$, which, in turn, is  much larger than the thickness $h$. That is, $\ell\gg \e\gg h$.
Since two characteristic dimensions are much smaller than the length, ribbons can be efficiently modelled as a one-dimensional continuum, see \cite{FF16}. In the literature, two types of one-dimensional models are found: rod models and ``Sadowsky type'' models. We shall mainly discuss
the latter, since we work within that framework; for rod type
models we refer to \cite{DA2015} and the references therein.

So far, ``Sadowsky type'' models have been deduced starting from a plate model, that is, from a two-dimensional model obtained from a three-dimensional problem by letting the thickness
$h$ go to zero. 
Starting from a Kirchhoff plate model, a one-dimensional model for an isotropic elastic ribbon was proposed by Sadowsky in 1930, \cite{HF2015b,Sadowsky1}. The model was formally justified in 1962 by Wunderlich, \cite{To2015, Wunder}, by considering the Kirchhoff model for a plate of length $\ell$ and width $\e$, and by letting $\e$ go to zero. The justification given was only formal,
since it was based on an ansatz on the deformation. Wunderlich's technique is quite ingenious, but it leads to a singular energy density; we refer to \cite{KF2015} for a rigorous analysis
of the so-called Wunderlich energy. A corrected Sadowsky type of energy was derived in \cite{FrHoMoPa} and generalized in \cite{AgDeKo, FrHoMoPa2}.

A third approach, which partly justifies the two approaches mentioned above, is to let  the width $\e$ and  the thickness $h$ go to zero simultaneously. By appropriately tuning the rates at which $\e$ and $h$ converge to zero, one obtains a hierarchy of one-dimensional models:
in \cite{FrMoPa, FrMoPa2} several rod models have been deduced, and in a forthcoming paper we will show that also ``Sadowsky type'' models can be obtained.

Before describing the contents of the present paper, we point out that the literature
on ribbons is really blooming in several interesting directions, see, for instance, \cite{AEK11, BH2015, Ch2014, CBK93, Ef2015, KTM12, SYU11, SH2015}.

Our starting point are the von K{\'a}rm{\'a}n plate models, whereas the papers quoted above have the Kirchhoff plate model as a starting point.
The  von K{\'a}rm{\'a}n model for plates has been successfully used in \cite{ChDeDa} to describe the plethora of morphological instabilities observed in a stretched and twisted ribbon.

The von K{\'a}rm{\'a}n plate equations, formulated more than a hundred years ago \cite{vK1910},
have been recently justified by Friesecke, James, and M\"uller~\cite{FrJaMu}.
These authors consider a three-dimensional non-linear hyper-elastic material in a reference configuration $\Omega_h=S_\e\times(-\frac{h}{2},\frac{h}{2})$ with a stored energy density $W:\R^{3\times 3}\to [0,+\infty)$ satisfying standard regularity and growth conditions.
In~\cite{FrJaMu} the set $S_\e$ is quite general, but in this introduction, in order
to be consistent with the previous discussion, we take 
$S_\e=(-\ell/2,\ell/2)\times (-\e/2,\e/2)$.
Then, the energy associated with a deformation $y:\Omega_h\to\R^3$ is given by
$$
{\mathcal E}^h(y)=\int_{\Omega_h}W(\nabla y)\,dx.
$$
By scaling the elastic energy per unit volume ${\mathcal E}^h/h\sim h^\beta$, with $\beta$ a positive real parameter, in~\cite{FrJaMu} a hierarchy of plate models has been
derived (by letting $h$ go to zero) by means of $\Gamma$-convergence theory.
The larger  $\beta$ is, the smaller the energy becomes. Therefore, heuristically, for large $\beta$ the limit of the rescaled energy should produce the linear plate equation. This is indeed corroborated in~\cite{FrJaMu}. 
Still in the same paper it is shown that for $\beta=2\alpha-2$ and for the regimes  $\alpha>3$, $\a=3$, and $2<\a<3$ three different $\Gamma$-limits are obtained that correspond to  von K{\'a}rm{\'a}n type of energies.

Precisely, denoting by $u:S_\e\to\R^2$ and $v:S_\e\to\R$ the in-plane and the out-of-plane
  displacement fields, respectively, the three asymptotic energies are as follows
  (see  \cite[Theorem~2]{FrJaMu}):
\begin{itemize}
  \item[({\em LvK})$_\e$] for $\alpha>3$ we have the ``linearized'' von K{\'a}rm{\'a}n theory, where $u=0$ and $v$ minimizes the functional
  $$
  I^{LvK}_\e(v):=\frac{1}{24}\int_{S_\e} Q_2(\nabla^2v)\,dx,
  $$
with $Q_2:\R^{2\times 2}_{\sym} \to[0,+\infty)$ the positive definite quadratic form of linearized elasticity, see Remark~\ref{remark} for a precise definition;
  \item[({\em vK})$_\e$] for $\alpha=3$ we have the von K{\'a}rm{\'a}n theory, where the in-plane and the out-of-plane
  displacements $u$ and $v$ minimize the functional
  $$
  I^{vK}_{\e}(u,v):=\frac12\int_{S_\e} Q_2\Big(\frac12[\nabla u+(\nabla u)^T+\nabla v\otimes\nabla v]\Big)\,dx+\frac{1}{24}\int_{S_\e} Q_2(\nabla^2v)\,dx;
  $$
  \item[({\em CvK})$_\e$] for  $2<\alpha<3$ we have the ``constrained'' von K{\'a}rm{\'a}n theory, in which the functional
 $$
  I^{CvK}_\e(v):=\frac{1}{24}\int_{S_\e} Q_2(\nabla^2v)\,dx
  $$
has to be minimized under the non-linear constraint
  \begin{equation}\label{nlc}
  \nabla u+(\nabla u)^T+\nabla v\otimes\nabla v=0,
  \end{equation}
  or, equivalently, the functional $ I^{CvK}_\e$ has to be minimized under the constraint
  $$
  \det(\nabla^2v)=0
  $$
  (which  is, in turn,  necessary and sufficient for the existence of a map $u$ satisfying~\eqref{nlc}).
\end{itemize}
The existence of minimizers and the characterization of the Euler equations for constrained von K{\'a}rm{\'a}n plates have been studied in~\cite{Ho2014}.

By letting $h$ go to zero, the three-dimensional domain $\Omega_h=S_\e\times(-\frac{h}{2},\frac{h}{2})$ is ``squeezed'' to become $S_\e$.
In this paper, we consider the von K{\'a}rm{\'a}n energies and we let $\e$ go to zero, still by means of $\Gamma$-convergence, to find  one-dimensional models for elastic ribbons in the von K{\'a}rm{\'a}n regimes. In this way, the two-dimensional
domain $S_\e=(-\ell/2,\ell/2)\times (-\e/2,\e/2)$ is  ``squeezed'' to the segment 
$I=(-\ell/2,\ell/2)$,
which we parametrize with the coordinate $x_1$.
In the limit, the in-plane displacement $u : S_\e\to\R^2$ generates two displacements: an axial displacement $\xi_1 : I\to\R$, and an orthogonal ``in-plane'' displacement $\xi_2:I\to\R$.
The out-of-plane displacement $v:S_\e\to\R$, in turn, generates an ``out-of-plane'' displacement $w:I\to\R$ and the derivative of $v$ in the direction orthogonal to the axis
leads to a rotation $\vartheta:I\to \R$. 
The limit energies that we find in the three regimes are the following:
\begin{itemize}
  \item[({\em LvK})] the limit of the ``linearized'' von K{\'a}rm{\'a}n energy is
$$
J^{LvK}(w,\vartheta):=\frac{1}{24}\int_I Q_1(w'', \vartheta')\,dx_1;
$$
 \item[({\em vK})] the limit of the von K{\'a}rm{\'a}n energy is
$$
J^{vK}(\xi,w,\vartheta):=\frac12\int_I Q_0\Big(\xi_1'+\frac{|w'|^2}{2}\Big)\,dx_1+\frac{1}{24}\int_I 
\big(Q_0(\xi_2'')+Q_1(w'', \vartheta') \big)\,dx_1;
$$
  \item[({\em CvK})] the limit of the ``constrained'' von K{\'a}rm{\'a}n energy is
$$
J^{CvK}(w,\vartheta):=\frac{1}{24}\int_I \o Q(w'', \vartheta')\,dx_1.
$$
\end{itemize}
Here $Q_1, Q_0$, and $\overline Q$ are energy densities whose precise definition
can be found in Section~\ref{ns}; 
see Remark~\ref{remark} for the specialization of these energies in the isotropic case. 

We note that, since $Q_1$ is quadratic, the functional ({\em LvK}) corresponds to the energy of a linearly elastic ``three-dimensional beam'' in which the section $S_\e$ is unstretchable: the energy 
is simply due to the ``out-of-plane'' bending of the axis and to the torsion of the cross-section orthogonal to the axis. The limit functional ({\em vK}) is non-linear and penalizes stretching and both bendings of the axis, as well as the torsion of the cross-section.
The functional ({\em CvK}) is sometimes called the energy of a beam with large deflections, see \cite{Wa}. 
Despite the appearance, the energy functional
({\em CvK}) is very different from that of ({\em LvK}).
Indeed, in contrast to $Q_1$, the energy density $\o Q$ is not quadratic. It incorporates
into its definition the non-linear constraint \eqref{nlc} that appears into the two-dimensional model ({\em CvK})$_\e$.
The energy density $\o Q$ agrees with
the corrected Sadowsky energy density found in \cite{FrHoMoPa} in the isotropic case, and with that found in \cite{FrHoMoPa2} for the general anisotropic case.
To the best of our knowledge, the model ({\em CvK}) is new.

We conclude this introduction by pointing out that the statements of the results
and the precise definitions are given in  Section~\ref{ns}, while 
Section~\ref{proofs} is exclusively devoted to the the proofs of these results.

\section{Narrow strips}\label{ns}

Let $\ell> 0$, let $I$ denote the interval $(-\ell/2, \ell/2)$, and let
 $S_{\e} = I\times (-\e/2, \e/2)$ with
$\e>0$. For  $u\in W^{1,2}(S_\e;\R^2)$ and $v\in W^{2,2}(S_\e)$
we consider the scaled von K\'arm\'an extensional and bending  energies
$$
{\mathscr J}^{ext}_{\e}(u,v) =\frac{1}{\e}\frac12\int_{S_\e} Q_2\Big(Eu+\frac12\nabla v\otimes\nabla v\Big)\,dx,\qquad
{\mathscr J}^{ben}_{\e}(v) = \frac{1}{\e}\frac{1}{24}\int_{S_{\e}} Q_2(\nabla^2 v)\, dx,
$$
where $Eu =\frac12({\nabla u +\nabla u^T})$ is the symmetric part of the gradient of the in-plane displacement $u$,  while $\nabla^2 v$
denotes the Hessian matrix of the out-of-plane displacement $v$. The energy density 
$Q_2:\R^{2\times 2}_{\sym} \to[0,+\infty)$ is assumed to be a positive definite quadratic form.

To simplify our analysis we rewrite the energies over the domain $S := S_1=I\times (-1/2, 1/2)$. More precisely, we 
introduce the scaled versions $y:S\to\R^2$ and $w: S\to\R$ of $u$ and $v$, respectively, by setting
$$
y_1(x_1, x_2) := u_1(x_1, \e x_2),\quad y_2(x_1, x_2) := \e u_2(x_1, \e x_2),\quad w(x_1, x_2) := v(x_1, \e x_2),
$$
and define the scaled differential operators
$$
 E^\e y:=\begin{pmatrix}
\partial_{1}y_1 & \frac{1}{2\e}(\partial_{1}y_2+\partial_{2}y_1)
\\[1ex]
\frac{1}{2\e}(\partial_{2}y_1+\partial_{1}y_2)& \frac{1}{\e^2}\partial_{2}y_2
\end{pmatrix},
$$
\vspace{1ex}
$$
 \De w:= \Big(\d_1w, \frac{1}{\e}\d_2w\Big),\quad  \nabla^2_\e w:=\begin{pmatrix}
\partial^2_{11}w & \frac{1}{\e}\partial^2_{12}w
\\[1ex]
\frac{1}{\e}\partial^2_{21}w & \frac{1}{\e^2}\partial^2_{22}w
\end{pmatrix},$$ 
so that
$$
E^\e y(x) = Eu (x_1,\e x_2),\quad \nabla_\e w(x)=\nabla v(x_1,\e x_2),\quad \nabla^2_\e w(x) = \nabla^2 v(x_1, \e x_2).
$$
By performing the change of variables in the energy integrals
we have that ${\mathscr J}^{ext}_{\e}(u,v)=J^{ext}_{\e}(y,w)$ and ${\mathscr J}^{ben}_{\e}(v)=J^{ben}_{\e}(w)$, where
\begin{equation}\label{jeps}
J^{ext}_{\e}(y,w) :=\frac12\int_{S} Q_2\Big(E^\e y+\frac12\nabla_\e w\otimes\nabla_\e w\Big)\,dx,\qquad J^{ben}_{\e}(w) := \frac{1}{24}\int_S Q_2(\nabla^2_\e w)\, dx.
\end{equation}

Since we do not impose boundary conditions, we require the displacements to have zero average and,
for the out-of plane component, also zero average gradient.
That is, we shall work in the following spaces: for every open  set $\Omega\subset \R^\a$ with $\a=1,2$, we consider
\begin{align*}
W^{1,2}_{\langle 0\rangle}(\Omega)&:=\Big\{g\in W^{1,2}(\Omega): \int_\Omega g(x)\, dx=0\Big\},\\
W^{2,2}_{\langle 0\rangle}(\Omega)&:=\Big\{g\in W^{2,2}(\Omega): \int_\Omega g(x)\, dx=0 \mbox{ and } \int_\Omega \nabla g(x)\, dx=0\Big\},
\end{align*}
and similarly we define $W^{1,2}_{\langle 0\rangle}(\Omega;\R^2)$.

Our first result is  about compactness of sequences with bounded energy;
the limit of the in-plane displacements will belong to   
the space of two-dimensional Bernoulli-Navier functions 
defined by
\begin{align*}
BN_{\langle 0\rangle}(S;\R^2):&=\{g\in W^{1,2}_{\langle 0\rangle}(S;\R^2)\ :\ (Eg)_{12}=(Eg)_{22}=0\}  \\
&=  \{g\in W^{1,2}_{\langle 0\rangle}(S;\R^2)\ :\ \exists\,\xi_1\in W^{1,2}_{\langle 0\rangle}(I)\mbox{ and }\xi_2\in W^{1,2}_{\langle 0\rangle}(I)\cap W^{2,2}(I)
\mbox{ such that }\\
&\hspace{24ex}g_1(x)=\xi_1(x_1)-x_2\xi_2'(x_1),\,g_2(x)=\xi_2(x_1)\}, 
\end{align*}
where the second characterization can be obtained by arguing as in \cite[Section~4.1]{LeDret}. 

\begin{lemma}\label{VKcompactness}
Let $(w_{\e})\subset W^{2,2}_{\langle 0\rangle}(S)$ be a sequence such that
\begin{equation}\label{VKcp-1}
\sup_{\e} J^{ben}_{\e}(w_{\e}) < \infty.
\end{equation}
Then, up to a subsequence, there exist a {\rm vertical displacement} $w\in W^{2,2}_{\langle 0\rangle}(I)$ and
a {\rm twist function} $\vartheta\in W^{1,2}_{\langle 0\rangle}(I)$ such that
\begin{equation}\label{VKconv}
w_{\e}\weak w \mbox{ in }W^{2,2}(S),\qquad
\D_{\e} w_{\e} \weak (w', \vartheta)\mbox{ in }W^{1,2}(S; \R^2),
\end{equation}
and
\begin{equation}
\label{VKconv_second}
\nabla^2_\e w_\e \weak
\begin{pmatrix}
w'' & \vartheta'
\\
\vartheta' & \gamma
\end{pmatrix}
\mbox{ in }L^{2}(S; \R^{2\times 2}_{\sym})
\end{equation}
for a suitable $\gamma\in L^2(S)$.

Moreover, if $(y_{\e})\subset  W^{1,2}_{\langle 0\rangle}(S;\R^2)$  is a further sequence such that
\begin{equation}\label{VKcp-2}
\sup_{\e} J^{ext}_{\e}(y_{\e},w_\e) < \infty,
\end{equation}
then, up to a subsequence, there exists $y\in BN_{\langle 0\rangle}(S;\R^2)$
 such that
\begin{equation*}\label{VKcony}
y_{\e}\weak y \mbox{ in }W^{1,2}(S;\R^2).
\end{equation*}
Also,  
$$
E^\e y_\e \weak
E
\mbox{ in }L^{2}(S; \R^{2\times 2}_{\sym})
$$
for a suitable $E\in L^{2}(S; \R^{2\times 2}_{\sym})$ such that $E_{11}=\partial_1y_1$.
\end{lemma}

The rest of this section is devoted to state the $\Gamma$-convergence results starting from the simpler case of the linearized theory ({\em LvK})$_\e$,
and proceeding in the order of increasing difficulty to consider the standard and the constrained models ({\em vK})$_\e$ and ({\em CvK})$_\e$, respectively.

\subsection{The linearized von K{\'a}rm{\'a}n model}

In order to state our first convergence result we need to introduce some definitions.
Let $Q_1:\R\times\R \to [0,+\infty)$ be defined by
$$
Q_1(\kappa,\tau) : = \min_{\gamma\in \R} \Big\{ Q_2(M)\ :\
M=
\begin{pmatrix}
\kappa & \tau
\\
\tau & \gamma
\end{pmatrix}
\Big\}.
$$
Let $J^{LvK}:W^{2,2}_{\langle 0\rangle}(I)\times W^{1,2}_{\langle 0\rangle}(I) \to \R$ be defined by
$$
J^{LvK}(w,\vartheta):=\frac{1}{24}\int_I Q_1(w'', \vartheta')\,dx_1.
$$

\begin{theorem}\label{VKlinGamma}
As $\e\to 0$, the functionals $J_\e^{ben}$ $\Gamma$-converge to the functional
$J^{LvK}$ in the following sense:
\begin{enumerate}
\item[(i)] {\rm (liminf inequality)} for every sequence $(w_{\e})\subset W^{2,2}_{\langle 0\rangle}(S)$,
$w\in W^{2,2}_{\langle 0\rangle}(I)$, and $\vartheta\in W^{1,2}_{\langle 0\rangle}(I)$ such that $w_{\e}\weak w$ in $W^{2,2}(S)$, and $\D_{\e} w_{\e} \weak (w', \vartheta)$ in $W^{1,2}(S; \R^2)$,
we have that
$$
\liminf_{\e\to 0} J_\e^{ben}(w_\e)\ge J^{LvK}(w,\vartheta);
$$
\item[(ii)] {\rm (recovery sequence)} for every $w\in W^{2,2}_{\langle 0\rangle}(I)$ and $\vartheta\in W^{1,2}_{\langle 0\rangle}(I)$ there exists a sequence
$(w_{\e})\subset W^{2,2}_{\langle 0\rangle}(S)$ such that
$w_{\e}\weak w$ in $W^{2,2}(S)$, $\D_{\e} w_{\e} \weak (w', \vartheta)$ in $W^{1,2}(S; \R^2)$, and
$$
\limsup_{\e \to 0} J_\e^{ben}(w_\e)\le J^{LvK}(w,\vartheta).
$$
\end{enumerate}
\end{theorem}

\subsection{The von K{\'a}rm{\'a}n model}

The statement of our second convergence result needs some further definitions.
Let $Q_0:\R \to [0,+\infty)$ be defined by
$$
Q_0(\mu) : =\min_{z\in\R}Q_1(\mu,z)=\min_{(z_1,z_2)\in \R^2} \Big\{ Q_2(M)\ :\
M=
\begin{pmatrix}
\mu & z_1
\\
z_1 & z_2
\end{pmatrix}
\Big\}.
$$
Let $J^{vK}:BN_{\langle 0\rangle}(S;\R^2)\times W^{2,2}_{\langle 0\rangle}(I)\times W^{1,2}_{\langle 0\rangle}(I) \to \R$ be defined by
$$
J^{vK}(y,w,\vartheta):=\frac12\int_S Q_0\Big(\partial_1y_1+\frac{|w'|^2}{2}\Big)\,dx+\frac{1}{24}\int_I Q_1(w'', \vartheta')\,dx_1.
$$

\begin{theorem}\label{VKusGamma}
As $\e\to 0$, the functionals $J_\e^{vK}:=J^{ext}_\e+J_\e^{ben}$ $\Gamma$-converge to the functional $J^{vK}$ in the following sense:
\begin{enumerate}
\item[(i)] {\rm (liminf inequality)} for every pair of sequences $(y_\e)\subset W^{1,2}_{\langle 0\rangle}(S;\R^2)$, $(w_{\e})\subset W^{2,2}_{\langle 0\rangle}(S)$, $y\in BN_{\langle 0\rangle}(S;\R^2)$,
$w\in W^{2,2}_{\langle 0\rangle}(I)$, and $\vartheta\in W^{1,2}_{\langle 0\rangle}(I)$ such that $y_{\e}\weak y$ in $W^{1,2}(S;\R^2)$, $w_{\e}\weak w$ in $W^{2,2}(S)$, and $\D_{\e} w_{\e} \weak (w', \vartheta)$ in $W^{1,2}(S; \R^2)$,
we have that
$$
\liminf_{\e\to 0} J_\e^{vK}(y_\e,w_\e)\ge J^{vK}(y,w,\vartheta);
$$
\item[(ii)] {\rm (recovery sequence)} for every $y\in BN_{\langle 0\rangle}(S;\R^2)$, $w\in W^{2,2}_{\langle 0\rangle}(I)$ and $\vartheta\in W^{1,2}_{\langle 0\rangle}(I)$ there exists a pair of sequences  $(y_\e)\subset W^{1,2}_{\langle 0\rangle}(S;\R^2)$,
$(w_{\e})\subset W^{2,2}_{\langle 0\rangle}(S)$ such that $y_{\e}\weak y$ in $W^{1,2}(S;\R^2)$,
$w_{\e}\weak w$ in $W^{2,2}(S)$, $\D_{\e} w_{\e} \weak (w', \vartheta)$ in $W^{1,2}(S; \R^2)$, and
$$
\limsup_{\e \to 0} J_\e^{vK}(y_\e,w_\e)\le J^{vK}(y,w,\vartheta).
$$
\end{enumerate}
\end{theorem}

\subsection{The constrained von K{\'a}rm{\'a}n model}

The constrained von K{\'a}rm{\'a}n energy of a displacement $v\in W^{2,2}_{\langle 0\rangle}(S_\e)$ such that $\det \nabla^2 v=0 \text{ a.e.\ in } S_\e$ is ${\mathscr J}^{ben}_\e(v)$. 
We observe that
the map $w$, defined over the rescaled domain, belongs to the space
$$
\We:= \big\{ w\in W^{2,2}_{\langle 0\rangle}(S): \  \det \nabla^2_\e w=0 \text{ a.e.\ in } S \big\}.
$$
We set $J^{CvK}_\e:\We\to\R$  the functional $J^{CvK}_\e(w)=J_\e^{ben}(w)$.

Let $\overline Q:\R\times \R\to [0,+\infty)$ be defined by
$$
\overline Q(\kappa,\tau) : = \min_{\gamma\in \R} \Big\{ Q_2(M)+\alpha^+ (\det M)^+ + \alpha^-(\det M)^-:
M=
\begin{pmatrix}
\kappa & \tau
\\
\tau & \gamma
\end{pmatrix}
\Big\},
$$
where
$$
\alpha^+:=\sup\{ \alpha>0: \ Q_2(M)+\alpha\det M\geq0 \text{ for every } M\in \R^{2\times 2}_{\sym}\}
$$
and
$$
\alpha^-:=\sup\{ \alpha>0: \ Q_2(M)-\alpha\det M\geq0 \text{ for every } M\in \R^{2\times 2}_{\sym}\}.
$$
Let $J^{CvK}:W^{2,2}_{\langle 0\rangle}(I)\times W^{1,2}_{\langle 0\rangle}(I) \to \R$ be defined by
$$
J^{CvK}(w,\vartheta):=\frac{1}{24}\int_I \o Q(w'', \vartheta')\,dx_1.
$$

\begin{theorem}\label{VKGamma}
As $\e\to 0$, the functionals $J_\e^{CvK}$  $\Gamma$-converge to the functional $J^{CvK}$ in the following sense:
\begin{enumerate}
\item[(i)] {\rm (liminf inequality)} for every sequence $(w_{\e})$ with $w_\e\in \We$,
$w\in W^{2,2}_{\langle 0\rangle}(I)$, and $\vartheta\in W^{1,2}_{\langle 0\rangle}(I)$ such that $w_{\e}\weak w$ in $W^{2,2}(S)$, and $\D_{\e} w_{\e} \weak (w', \vartheta)$ in $W^{1,2}(S; \R^2)$,
we have that
$$
\liminf_{\e\to 0} J_\e^{CvK}(w_\e)\ge J^{CvK}(w,\vartheta);
$$
\item[(ii)] {\rm (recovery sequence)} for every $w\in W^{2,2}_{\langle 0\rangle}(I)$ and $\vartheta\in W^{1,2}_{\langle 0\rangle}(I)$ there exists a sequence
$(w_{\e})$ with $w_\e\in \We$ such that
$w_{\e}\weak w$ in $W^{2,2}(S)$, $\D_{\e} w_{\e} \weak (w', \vartheta)$ in $W^{1,2}(S; \R^2)$, and
$$
\limsup_{\e \to 0} J_\e^{CvK}(w_\e)\le J^{CvK}(w,\vartheta).
$$
\end{enumerate}
\end{theorem}

\begin{remark}\label{remark}
The quadratic energy density $Q_2$ can be computed from the non-linear energy density $W$ of the material, also mentioned in the introduction,  by first computing the quadratic energy density $Q_3$, see \cite{FrJaMu},
$$
Q_3(F):=\frac{\partial^2W}{\partial F^2}(I)(F,F)=\sum_{i,j,k,l=1}^3
\frac{\partial^2W}{\partial F_{ij}\partial F_{kl}}(I)F_{ij}F_{kl} ,\qquad F\in \R^{3\times3},
$$
and then by minimizing over the third column and row:
$$
Q_2(A):= \min\{Q_3(F): F_{\alpha\beta}=A_{\alpha\beta} \quad \alpha,\beta=1,2 \}, \qquad A\in \R^{2\times 2}_{\sym}.
$$

If the energy density $W$ is isotropic, the quadratic energy density $Q_3$ has the following representation:
$$
Q_3(F)=2\mu|F_{\rm sym}|^2+\lambda (F_{\rm sym}\cdot I)^2,\qquad F_{\rm sym}:=\frac{F+F^T}{2}\in \R^{3\times3},
$$
where $\mu$ and $\lambda$ are the so-called Lam\'e coefficients. A simple computation then leads to
$$
Q_2(A)=2\mu|A|^2+\frac{2\mu\lambda}{2\mu+\lambda} (A\cdot I)^2,\qquad A\in \R^{2\times 2}_{\sym}.
$$
The energy densities $Q_1, Q_0$, and $\overline Q$, may be found to have the following representation
$$
Q_1(\kappa,\tau)=E_{\rm Y}\kappa^2+4\mu \tau^2,
$$
where $E_{\rm Y}:=\mu\frac{2\mu+3\lambda}{\mu+\lambda}$ is the Young modulus of the material,
$$
Q_0(\kappa)=E_{\rm Y}\kappa^2,
$$
and
$$
\frac 1{12}\overline Q(\kappa,\tau) = \begin{cases}
\mathcal{D}\dfrac{(\kappa^2+\tau^2)^2}{\kappa^2} & \text{ if } |\kappa|>|\tau|,
\smallskip\\
4\mathcal{D}\tau^2 & \text{ if } |\kappa|\leq|\tau|,
\end{cases}
$$
where $\mathcal{D}:=\frac{\mu(\lambda+\mu)}{3(2\mu+\lambda)}$ is the bending stiffness.

\end{remark}

\section{Proofs}\label{proofs}

This section is devoted to prove the theorems stated in the previous section. 
For a given function $u\in L^1(S)$, we shall
denote by $\langle u\rangle$ the integral mean value of $u$ on $S$, that is,
$$
\langle u\rangle:=\frac{1}{\ell}\int_S u(x)\,dx.
$$
We use the same notation to denote the average over $I$ of functions defined on $I$.

\begin{proof}[Proof of~Lemma~\ref{VKcompactness}]
Let $(w_{\e})\subset  W^{2,2}_{\langle 0\rangle}(S)$ be a sequence of vertical displacements of $S$ satisfying \eqref{VKcp-1}.
This bound and the fact that $Q_2$ is positive definite imply that
\begin{equation}\label{bound12}
\|\d_{11}^2 w_{\e}\|_{L^2(S)} + \|\e^{-1}\d^2_{12} w_{\e}\|_{L^2(S)} + \|\e^{-2}\d^2_{22} w_{\e}\|_{L^2(S)} \leq C
\end{equation}
for any $\e$. Since
$$
\int_S w_\e(x)\, dx=0, \qquad \int_S \nabla w_\e(x)\, dx=0
$$
for every $\e>0$, by Poincar\'e-Wirtinger inequality the sequence $(w_\e)$ is uniformly bounded in $W^{2,2}(S)$. Therefore, there exists $w\in W^{2,2}_{\langle 0\rangle}(S)$ such that
 $w^\e\weak w$ weakly in $W^{2,2}(S)$,  up to a subsequence. 

By the previous bound, $\nabla (\e^{-1}\d_2 w_\e)$ is a bounded sequence in $L^2(S; \R^2)$ and, by Poincar\'e-Wirtinger inequality, also $(\e^{-1}\d_2 w_\e)$ is bounded in $L^2(S)$. 
It follows that $w$ is independent of $x_2$ and there exixts $\vartheta\in W^{1,2}_{\langle 0\rangle}(S)$ such that $\e^{-1}\d_2 w_\e\weak \vartheta$
weakly in $W^{1,2}(S)$, up to a subsequence. Moreover, also $\vartheta$ is independent of $x_2$.

By \eqref{bound12}, up to subsequences, we have that $\nabla^2_\e w_\e$ converges to a matrix field $A$ weakly in $L^2(S;\R^{2\times2}_{\sym})$. By using the convergences established above,  it follows that
$A_{11}=w''$ and
$A_{12} =\vartheta'$. The entry $A_{22}$, that cannot be identified in terms of
$w$ and $\vartheta$, is denoted by $\gamma$ in the statement. This proves \eqref{VKconv_second}.

We now prove the second part of the statement. The bound~\eqref{VKcp-2} implies that
\begin{equation}\label{2b}
\Big\|E^\e y_\e +\frac12\nabla_\e w_\e\otimes\nabla_\e w_\e\Big\|_{L^2}\le C
  \end{equation}
for any $\e$. Since $(\nabla^2_\e w_\e)$ is bounded in $L^2$, we have
\begin{align*}
\|\nabla_\e w_\e\otimes\nabla_\e w_\e\|_{L^2}&\le C\| |\nabla_\e w_\e|^2\|_{L^2}=C\| \nabla_\e w_\e\|^2_{L^4}\le C(\|\partial_1 w_\e\|_{L^4}^2+\|\e^{-1}\d_2 w_\e\|_{L^4}^2)\\ &
\le
C(\|\nabla\partial_1 w_\e\|_{L^2}^2+\|\nabla\e^{-1}\d_2 w_\e\|_{L^2}^2)
\le C \| \nabla^2_\e w_\e\|^2_{L^2}\le C
\end{align*}
for any $\e$, and the third to last inequality follows by the imbedding $W^{1,2}(S)\subset L^q$ $\forall\,q\in[2,+\infty)$ and Poincar\'e-Wirtinger inequality. Together with~\eqref{2b}, this implies that the sequence $(E^\e y_\e)$ is bounded in $L^2$.

By  the definition of $E^\e$ and Korn-Poincar\'e inequality we have that
\begin{equation}\label{boundEye}
\|y_\e\|_{W^{1,2}}\le C\| Ey_\e\|_{L^2}\le C\| E^\e y_\e\|_{L^2}\le C.
\end{equation}
Hence, up to subsequences, there exist $E\in L^2(S;\R^{2\times 2}_{\rm sym})$ and $y\in W^{1,2}_{\langle 0\rangle}(S;\R^2)$
such that
\begin{eqnarray*}
E^\e y_\e\weak E&\mbox{ in }L^2(S;\R^{2\times 2}_{\rm sym}),\\
y_\e\weak y&\mbox{ in }W^{1,2}(S;\R^2).
\end{eqnarray*}
By the definition of $E^\e$ and \eqref{boundEye} we have that
$$
(Ey_\e)_{12}\weak 0=(Ey)_{12},\quad (Ey_\e)_{22}\weak 0=(Ey)_{22};
$$
hence, $y\in BN_{\langle 0\rangle}(S;\R^2)$. 
Finally, the observation that $(E^\e y_\e)_{11}=\partial_1 (y_\e)_1\weak \partial_1y_1$ in $L^2(S)$ concludes the proof.
\end{proof}

\begin{proof}[Proof of Theorem~\ref{VKlinGamma}--(i)]
Let $(w_{\e})\subset W^{2,2}_{\langle 0\rangle}(S)$ be such that $w_{\e}\weak w$ in $W^{2,2}(S)$, and $\D_{\e} w_{\e} \weak (w', \vartheta)$ in $W^{1,2}(S; \R^2)$,
for some $w\in W^{2,2}_{\langle 0\rangle}(I)$ and $\vartheta\in W^{1,2}_{\langle 0\rangle}(I)$. Without loss of generality, we can assume that $\liminf_{\e\to 0} J_{\e}^{ben}(w_{\e})<+\infty$.
By Lemma~\ref{VKcompactness} we infer that, up to subsequences,
$$
\nabla^2_\e w_\e \weak
\begin{pmatrix}
w'' & \vartheta'
\\
\vartheta' & \gamma
\end{pmatrix}=:M_\gamma
\mbox{ in }L^{2}(S; \R^{2\times 2}_{\sym})
$$
for some $\gamma\in L^2(S)$.
By weak lower semicontinuity  and the definition of $Q_1$ we have
\begin{align*}
\liminf_{\e\to 0} J_\e^{ben}(w_{\e}) &= \liminf_{\e\to 0}\frac{1}{24}\int_S Q_2(\nabla^2_{\e} w_\e)\,dx\\ 
&\geq
\frac{1}{24}\int_S Q_2(M_\gamma)\,dx\ge \frac{1}{24}\int_I Q_1(w'',\theta')\,dx_1=J^{LvK}(w,\theta).
\end{align*}
\end{proof}

\begin{proof}[Proof of Theorem~\ref{VKlinGamma}--(ii)]
Let $w\in W^{2,2}_{\langle 0\rangle}(I)$ and $\vartheta\in W^{1,2}_{\langle 0\rangle}(I)$. We set
$$
M_\gamma:= \left( \begin{array}{cc}
w'' & \vartheta' \\
\vartheta' & \gamma
\end{array}\right),
$$
where $\gamma\in L^2(I)$ is such that
$$
Q_1(w'', \vartheta')= Q_2(M_\gamma).
$$
The fact that $\gamma$ belongs to $L^2(I)$ follows immediately by choosing
$M_0=w'' e_1\otimes e_1+\vartheta'(e_1\otimes e_2+e_2\otimes e_1)$ as a competitor in the definition of $Q_1$
and by using the positive definiteness of $Q_2$.

Let $\theta_\e\in C^\infty({\overline I})$ be such that $\int_I\theta_\e(x_1)\,dx_1=0$, $\theta_\e\to\theta$ in $W^{1,2}(I)$, and $\e\theta_\e''\to0$ in $L^2(I)$. Let $\gamma_\e\in C^\infty({\overline I})$ be such that  $\gamma_\e\to\gamma$,  $\e\gamma_\e'\to0$ and $\e^2\gamma_\e''\to0$ in $L^2(I)$. Let
\begin{equation}\label{werec}
w_\e(x)=w(x_1)+\e x_2\theta_\e(x_1)+\frac{\e^2}{2}\big(x_2^2\gamma_\e(x_1)-\langle x_2^2 \gamma_\e\rangle-x_1\langle x_2^2\gamma_\e'\rangle\big).
\end{equation}
It turns out that $w_\e\in W^{2,2}_{\langle 0\rangle}(S)$ and, by the convergences  above, we have $w_\e\to w$ in $W^{2,2}(S)$, $\D_{\e} w_{\e} \to (w', \vartheta)$ in $W^{1,2}(S; \R^2)$, and $\nabla^2_\e w_\e\to M_\gamma$ in $L^2(S)$. Moreover, by strong continuity we have
$$
\lim_{\e\to0}J_\e^{ben}(w_\e)=\lim_{\e\to0}\frac{1}{24}\int_S Q_2(\nabla^2_\e w_\e)\,dx=\frac{1}{24}\int_I Q_1(w'',\theta')\,dx_1
=J^{LvK}(w,\theta).
$$
\end{proof}

\begin{proof}[Proof of Theorem~\ref{VKusGamma}--(i)]
Let $(y_\e)\subset W^{1,2}_{\langle 0\rangle}(S;\R^2)$, $(w_{\e})\subset W^{2,2}_{\langle 0\rangle}(S)$ be such that $y_{\e}\weak y$ in $W^{1,2}(S;\R^2)$, $w_{\e}\weak w$ in $W^{2,2}(S)$, and $\D_{\e} w_{\e} \weak (w', \vartheta)$ in $W^{1,2}(S; \R^2)$,
for some $y\in BN_{\langle 0\rangle}(S;\R^2)$, $w\in W^{2,2}_{\langle 0\rangle}(I)$ and $\vartheta\in W^{1,2}_{\langle 0\rangle}(I)$. As usual, we can assume that $\liminf_{\e\to 0} J_{\e}^{vK}(y_\e,w_{\e})<+\infty$ and
by Lemma~\ref{VKcompactness} we deduce that, up to subsequences,
$$
E^\e y_\e\weak E\ \mbox{ and }\ \nabla^2_\e w_\e \weak
\begin{pmatrix}
w'' & \vartheta'
\\
\vartheta' & \gamma
\end{pmatrix}=:M_\gamma
\mbox{ in }L^{2}(S; \R^{2\times 2}_{\sym})
$$
for some $E\in L^2(S;\R^{2\times 2}_{\rm sym})$ with $E_{11}=\partial_1y_1$, and $\gamma\in L^2(S)$. Moreover, by the convergences above,
$$
E^\e y_\e+\frac12\nabla_\e w_\e\otimes\nabla_\e w_\e\weak E+\frac12(w',\theta)\otimes(w',\theta)\mbox{ in }L^2(S;\R^{2\times 2}_{\rm sym}).
$$
Then, by lower semicontinuity, we have
\begin{eqnarray*}
\liminf_{\e\to 0} J_\e^{vK}(y_\e,w_\e)&\ge&\liminf_{\e\to 0}\frac12 J_\e(y_\e,w_\e)+\liminf_{\e\to 0}\frac{1}{24}J_\e^{\rm lin}(w_\e)\\
&=&\liminf_{\e\to 0}\frac12\int_{S} Q_2\Big(E^\e y_\e+\frac12\nabla_\e w_\e\otimes\nabla_\e w_\e\Big)\,dx+\liminf_{\e\to 0}\frac{1}{24}\int_S Q_2(\nabla^2_\e w_\e)\, dx\\
&\ge&\frac12\int_S Q_2\Big( E+\frac12(w',\theta)\otimes(w',\theta)\Big)\,dx+\frac{1}{24}\int_S Q_2(M_\gamma)\, dx\\
&\ge&\frac{1}{2} \int_S Q_0\Big(\partial_1y_1+\frac{1}{2}|w'|^2\Big)\,dx+\frac{1}{24} \int_I Q_1(w'',\vartheta')\, dx_1\\
&=& J^{vK}(y,w,\vartheta).
\end{eqnarray*}
\end{proof}

\begin{proof}[Proof of Theorem~\ref{VKusGamma}--(ii)]
Let $y\in BN_{\langle 0\rangle}(S;\R^2)$, $w\in W^{2,2}_{\langle 0\rangle}(I)$, and $\vartheta\in W^{1,2}_{\langle 0\rangle}(I)$. As before, there exists
$\gamma\in L^2(I)$ such that the matrix
$$
M_\gamma:= \left( \begin{array}{cc}
w'' & \vartheta' \\
\vartheta' & \gamma
\end{array}\right)
$$
satisfies
$$
Q_1(w'', \vartheta')= Q_2(M_\gamma).
$$
There exist $\xi_1\in W^{1,2}_{\langle 0\rangle}(I)$ and $\xi_2\in W^{1,2}_{\langle 0\rangle}(I)\cap W^{2,2}(I)$ such that
$y_1(x)=\xi_1(x_1)-x_2\xi_2'(x_1)$ and  $y_2(x)=\xi_2(x_1)$.
Moreover, there exists $z\in L^2(S;\R^2)$ such that
the matrix
$$
M_z:= \left( \begin{array}{cc}
\partial_1y_1+\frac{1}{2}|w'|^2 & z_1 \\
z_1 & z_2
\end{array}\right)=
\left( \begin{array}{cc}
\xi_1'(x_1)-x_2\xi_2''(x_1)+\frac{1}{2}|w'(x_1)|^2 & z_1 \\
z_1 & z_2
\end{array}\right)
$$
satisfies
$$
 Q_0\Big(\partial_1 y_1+\frac{1}{2}|w'|^2\Big)=Q_2(M_z).
$$
It is easily seen that $z_1$ and $z_2$ depend linearly on $\partial_1 y_1+\frac{1}{2}|w'|^2$. Since 
$\partial_1 y_1(x_1,x_2)+\frac{1}{2}|w'(x_1)|^2=\xi_1'(x_1)-x_2\xi_2''(x_1)+\frac{1}{2}|w'(x_1)|^2$,
there exist $\zeta_\alpha\in L^2(I)$ and $\eta_\alpha\in L^2(I)$ such that  
$$
z_\alpha(x_1,x_2)=\zeta_\alpha(x_1)+x_2\eta_\alpha(x_1),\qquad\alpha=1,2. 
$$

 Let $w_\e$ be as  in the proof of Theorem~\ref{VKlinGamma}--(ii) (see \eqref{werec}), and let
$\zeta^\e_\alpha,\,\eta^\e_\alpha\in C^\infty({\overline I})$ be such that $\zeta^\e_\alpha\to \zeta_\alpha$ and $\eta^\e_\alpha\to \eta_\alpha$ in $L^{2}(I)$ and $\e {\zeta^\e_\alpha}^\prime\to0$ and $\e {\eta^\e_\alpha}^\prime\to0$ in $L^2(I)$.
Let us define
\begin{eqnarray*}
(y_\e)_1(x_1,x_2)&:=&\xi_1(x_1)-x_2\xi_2'(x_1)+\e\big(x_2^2\eta^\e_1(x_1)-\langle x_2^2\eta^\e_1\rangle\big), \\
(y_\e)_2(x_1,x_2)&:=&\xi_2(x_1)+\e\Big(\int_0^{x_1} \big(2\zeta^\e_1(s)-w'(s)\theta(s)\big)\,ds
-\langle\int_0^{x_1} \big(2\zeta^\e_1(s)-w'(s)\theta(s)\big)\,ds\rangle\Big)\\
 &&+\frac{\e^2}{2}{x_2}\big(2\zeta^\e_2(x_1)- \theta^2(x_1)\big)  +\frac{\e^2}{2}\big(x_2^2\eta^\e_2(x_1)-\langle x_2^2\eta^\e_2\rangle\big) .
\end{eqnarray*}
Then it is easy to check that
$$
E^\e y_\e+\frac12\nabla_\e w_\e\otimes\nabla_\e w_\e\to M_z\mbox{ in }L^2(S;\R^{2\times 2}_{\rm sym}).
$$
Thus, by strong continuity,
\begin{eqnarray*}
\lim_{\e\to 0} J_\e^{vK}(y_\e,w_\e)&=&\lim_{\e\to 0}\Big(\frac12 J_\e^{ext}(y_\e,w_\e)+\frac{1}{24}J_\e^{ben}(w_\e)\Big)\\
&=&\lim_{\e\to 0}\Big(\frac12\int_{S} Q_2\Big(E^\e y_\e+\frac12\nabla_\e w_\e\otimes\nabla_\e w_\e\Big)\,dx+\frac{1}{24}\int_S Q_2(\nabla^2_\e w_\e)\, dx\Big)\\
&=&\frac12\int_S Q_2( M_z)\,dx+\frac{1}{24}\int_S Q_2(M_\gamma)\, dx\\
&=&\frac{1}{2} \int_S Q_0\Big(\partial_1y_1+\frac{1}{2}|w'|^2\Big)\,dx+\frac{1}{24} \int_I Q_1(w'',\vartheta')\, dx_1\\
&=& J^{vK}(y,w,\vartheta).
\end{eqnarray*}
\end{proof}

The proof of the $\Gamma$-convergence theorem~\ref{VKGamma}
is based on a relaxation result for a quadratic integral functional with a constraint on the determinant, that has been proved in~\cite[Proposition~9]{FrHoMoPa2} and recalled here for reader's convenience.

Let $\mathcal B$ be a bounded open subset of $\R^n$.
Let $Q:\mathcal B\times \R^{2\times 2}_{\sym}\to[0,+\infty)$ be measurable in the first variable and quadratic in the second. Define the functional
$$
\FF : L^2\left(\mathcal B;\R^{2\times 2}_{\sym}\right)\to[0,+\infty]
$$
by
$$
\FF(M) :=
\begin{cases}
\displaystyle\int_{\mathcal B} Q(x,M(x))\, dx & \text{ if $\det M = 0$ a.e.\ in } \mathcal B,
\smallskip
\\
+ \infty & \text{ otherwise.}
\end{cases}
$$

\begin{theorem}[\cite{FrHoMoPa2}]\label{lsh}
The weak-$L^2$ lower semicontinuous envelope of $\FF$ is the functional
$$
\o\FF : L^2\left(\mathcal B;\R^{2\times 2}_{\sym}\right)\to[0,+\infty)
$$
given by
$$
\o\FF(M) = \int_{\mathcal B} \left( Q(x,M(x)) + \alpha^+(x) (\det M(x))^+ +\alpha^-(x) (\det M(x))^-  \right)\, dx,
$$
where for every $x\in \mathcal B$
$$
\alpha^+(x):=\sup\{ \alpha>0: \ Q(x,M)+\alpha\det M\geq0 \text{ for every } M\in \R^{2\times 2}_{\sym}\}
$$
and
$$
\alpha^-(x):=\sup\{ \alpha>0: \ Q(x,M)-\alpha\det M\geq0 \text{ for every } M\in \R^{2\times 2}_{\sym}\}.
$$
\end{theorem}

\begin{proof}[Proof of Theorem~\ref{VKGamma}--(i)]
Let $(w_{\e})$ be such that $w_\e\in \We$,
$w_{\e}\weak w$ in $W^{2,2}(S)$, and $\D_{\e} w_{\e} \weak (w', \vartheta)$ in $W^{1,2}(S; \R^2)$,
for some $w\in W^{2,2}_{\langle 0\rangle}(I)$ and $\vartheta\in W^{1,2}_{\langle 0\rangle}(I)$. Under the assumption that $\liminf_{\e\to 0} J_{\e}^{vK}(w_{\e})<+\infty$,
by Lemma~\ref{VKcompactness} we deduce that, up to subsequences,
$$
\nabla^2_\e w_\e \weak
\begin{pmatrix}
w'' & \vartheta'
\\
\vartheta' & \gamma
\end{pmatrix}
\mbox{ in }L^{2}(S; \R^{2\times 2}_{\sym})
$$
for some $\gamma\in L^2(S)$.
Since $\det \nabla^2_{\e} w_\e = 0$, an application of Theorem~\ref{lsh} with $Q(x,M):=Q_2(M)$ and ${\mathcal B}=S$ shows that
$$
\liminf_{\e\to 0} J_{\e}^{CvK}(w_{\e}) = \liminf_{\e\to 0} \frac{1}{24} \F(\nabla^2_{\e} w_\e) \geq \frac{1}{24}\o\F\begin{pmatrix}
w'' & \vartheta'
\\
\vartheta' & \gamma
\end{pmatrix}
 \geq \frac{1}{24}\int_I \overline Q(w'', \vartheta')\, dx_1.
$$
Therefore, we conclude that
$$
\liminf_{\e\to 0} J_\e^{CvK}(w_\e)\ge J(w,\vartheta).
$$
\end{proof}

\begin{proof}[Proof of Theorem~\ref{VKGamma}--(ii)]
Let $w\in W^{2,2}_{\langle 0\rangle}(I)$ and $\vartheta\in W^{1,2}_{\langle 0\rangle}(I)$.
We set
$$
M:= \left( \begin{array}{cc}
w'' & \vartheta' \\
\vartheta' & \gamma
\end{array}\right),
$$
where $\gamma\in L^2(I)$ is such that
$$
\overline Q(w'', \vartheta')= Q_2(M) +\alpha^+(\det M)^++\alpha^-(\det M)^-.
$$
As before, the fact that $\gamma$ belongs to $L^2(I)$ follows immediately by choosing
$M_0=w'' e_1\otimes e_1+\vartheta'(e_1\otimes e_2+e_2\otimes e_1)$ as a competitor in the definition of $\o Q$
and by using the positive definiteness of $Q_2$.

By Theorem~\ref{lsh} with $Q(x,M):=Q_2(M)$ and ${\mathcal B}=I$, there exist $M^j\in L^2(I; \R^{2\times 2}_{\sym})$
with $\det M^j = 0$ and such that
$M^j\weak M$ weakly in $L^2(I;\R^{2\times2}_{\sym})$ and $\F(M^j)\to \o\F(M)$, as $j\to\infty$.
We can also assume that $M^j\in C^\infty(\bar{I};\R^{2\times 2}_{\sym})$. The proof of this fact
relies on a construction described in \cite[Theorem~2.2--(ii)]{FrHoMoPa}. We give here full details for convenience of the reader. Suppose that $(M_n)$ be a sequence of matrices with the same properties of $(M^j)$ apart from the regularity, and denote by $\lambda_n \in L^2(I)$ the trace of $M_n$. Since $M_n$ is symmetric
with $\det M_n = 0$, there exists $\beta_n=\beta_n(x_1)\in (-\pi/2,\pi/2]$ such that
$$
M_n=
\left(
\begin{array}{cc}
\cos \beta_n & -\sin\beta_n\\
\sin \beta_n & \cos\beta_n
\end{array}
\right)
\left(
\begin{array}{cc}
\lambda_n & 0\\
0 & 0
\end{array}
\right)
\left(
\begin{array}{cc}
\cos \beta_n & \sin\beta_n\\
-\sin \beta_n & \cos\beta_n
\end{array}
\right),
$$
and $\beta_n$ is uniquely determined if $\lambda_n\neq 0$.
When $\lambda_n(x_1) = 0$, we set $\beta_n(x_1)=0$.
We may
assume without loss of generality that $\lambda_n\in L^{\infty}(I)$, possibly after truncating $\lambda_n$ in modulus by $n$,
while $M_n$ still enjoys the same properties as before.
We can find
$\lambda_{n,k}\in C^\infty(\bar I)$ and
$\beta_{n,k}\in C^\infty(\bar I;(-\pi/2,\pi/2))$ such that, as $k\to\infty$,
$\lambda_{n,k}\to\lambda_n$ and $\beta_{n,k}\to\beta_n$ in $L^p(I)$ for every $p<+\infty$.
Set
$$
M_{n,k}:=
\left(
\begin{array}{cc}
\cos \beta_{n,k} & -\sin\beta_{n,k}\\
\sin \beta_{n,k} & \cos\beta_{n,k}
\end{array}
\right)
\left(
\begin{array}{cc}
\lambda_{n,k} & 0\\
0 & 0
\end{array}
\right)
\left(
\begin{array}{cc}
\cos \beta_{n,k} & \sin\beta_{n,k}\\
-\sin \beta_{n,k} & \cos\beta_{n,k}
\end{array}
\right).
$$
Then, $\det M_{n,k}=0$ for every $n,k$ and $M_{n,k}\to M_n$ in $L^2(I;\R^{2\times2}_{\sym})$, as $k\to\infty$.

Thus, by a diagonal argument, we may assume that there exist $\lambda^j\in C^\infty(\bar I)$
and $\beta^j\in C^\infty(\bar I)$ such that $|\beta^j| < \pi/2$
on $\bar I$, and with
\begin{eqnarray*}
M^j
&:=&
\left(
\begin{array}{cc}
\cos \beta^j & -\sin\beta^j\\
\sin \beta^j & \cos\beta^j
\end{array}
\right)
\left(
\begin{array}{cc}
\lambda^j & 0\\
0 & 0
\end{array}
\right)
\left(
\begin{array}{cc}
\cos \beta^j & \sin\beta^j\\
-\sin \beta^j & \cos\beta^j
\end{array}
\right)\\[2pt]
&=&
\lambda^j\left(
\begin{array}{cc}
\cos^2 \beta^j & \sin\beta^j \cos\beta^j\\
\sin \beta^j \cos\beta^j & \sin^2\beta^j
\end{array}
\right)
\end{eqnarray*}
we have that $M^j\in C^\infty(\bar{I};\R^{2\times 2}_{\sym})$, $\det M^j=0$ for every $j$, $M^j\weak M$ in $L^2(I;\R^{2\times2}_{\sym})$, and
$\F(M_j)\to \overline\F(M)$, as $j\to\infty$.

For all $j = 1, 2, \dots$ and all $k$, $l\in \{1, 2\}$ we
define $\o M^j_{kl}(x_1) := \int_0^{x_1} M^j_{kl}(s)\, ds$,
$$
w^j(x_1):=\int_0^{x_1}(x_1-s)M^j_{11}(s)\, ds -
\frac{1}{\ell}\int_I \Big( \int_0^t (t - s) M^j_{11}(s)\, ds \Big)\, dt
-x_1\langle \o M^j_{11}\rangle,
$$
and
$$
\vartheta^j(x_1):=\o M^j_{12}(x_1) -
\langle \o M^j_{12} \rangle.
$$
It is clear that $w^j\weak w$ weakly in $W^{2,2}(I)$ and $\vartheta^j\weak \vartheta$ weakly in $W^{1,2}(I)$,
as $j\to\infty$. Moreover, $w^j\in W^{2,2}_{\langle 0\rangle}(I)$
and $\vartheta^j\in W^{1,2}_{\langle 0\rangle}(I)$.

After extending $\beta^j$ smoothly to all of $\R$, still satisfying $|\beta^j|<\pi/2$, we define
$\alpha^j := \frac \pi{2}+\beta^j$,
$$
\t b^j(\xi_1) := \cos \alpha^j(\xi_1) e_1 + \sin \alpha^j(\xi_1) e_2
\ \mbox{ and }\
\Phi^j(\xi_1,\xi_2) := \xi_1 e_1 + \xi_2 \t b^j(\xi_1).
$$
Observe that, by the definition of $\t b^j$ and since
$(w^j)'' = M_{11}^j$ and $(\theta^j)' = M_{12}^j$,
\begin{equation}
\label{konstr-1}
\begin{pmatrix}
(w^j)''
\\
(\theta^j)'
\end{pmatrix}\cdot \t b^j = 0.
\end{equation}
Arguing  as in \cite[Lemma~12]{FrHoMoPa2}, we
see that for every $\e\le\e^j$ the matrix $\D\Phi^j(\xi_1, \xi_2)$
is invertible for $|\xi_2|\leq\e$,
and the map $(\Phi^j)^{-1} : S_\e\to \R^2$ is well defined.
For such $\e$ define $z^j : S_{\e}\to\R$ by setting
\begin{equation}
\label{defvonz}
z^j\left( \Phi^j(\xi_1, \xi_2) \right)
=
w^j(\xi_1) + \xi_2 \t b^j(\xi_1)\cdot
\begin{pmatrix}
(w^j)'(\xi_1)
\\
\theta^j(\xi_1)
\end{pmatrix}.
\end{equation}
We clearly have
\begin{equation}
\label{VKz}
z^j(\cdot, 0) = w^j.
\end{equation}
Moreover, taking derivatives in \eqref{defvonz}
and using \eqref{konstr-1}, we obtain
\begin{align*}
\D z^j(\Phi^j)^T\D\Phi^j &= 
\left( (w^j)'(\xi_1) + \xi_2 (\t b^j)'(\xi_1)\cdot
\begin{pmatrix}
(w^j)'(\xi_1)
\\
\theta^j(\xi_1)
\end{pmatrix},\
\t b^j(\xi_1)\cdot
\begin{pmatrix}
(w^j)'(\xi_1)
\\
\theta^j(\xi_1)
\end{pmatrix}
\right)
\\
&= \begin{pmatrix}
(w^j)'(\xi_1)
\\
\theta^j(\xi_1)
\end{pmatrix}^T\D\Phi^j.
\end{align*}
Since $\D\Phi^j$ is invertible for small $|\xi_2|$, we conclude that
for small $|\xi_2|$ and all $\xi_1$
\begin{equation}
\label{konstruktion-2}
\D z^j\left( \Phi^j(\xi_1, \xi_2) \right) = 
\begin{pmatrix}
(w^j)'(\xi_1)
\\
\theta^j(\xi_1)
\end{pmatrix}.
\end{equation}
Taking the derivative with respect to $\xi_2$, we conclude that
$$
\D^2 z^j(\Phi^j(\xi_1, \xi_2))\,\t b^j(\xi_1) = 0.
$$
In particular, the kernel is nontrivial, so
$
\det\D^2 z^j = 0\mbox{ on }S_{\e}.
$
Since $M^j\,\t b^j = 0$, in particular we have that $(\D^2 z^j(\cdot, 0) - M^j)\,\t b^j = 0$.
But from \eqref{VKz} we see that
$$
e_1\cdot \left( \D^2 z^j(\cdot, 0) - M^j \right)e_1 = \d^2_{11} z^j(\cdot, 0) - (w^j)'' = 0.
$$
Since $e_2\cdot\t b^j\neq 0$ and 
since $(\D^2 z^j(\cdot, 0) - M^j)$ is symmetric, we conclude that 
\begin{equation}
\label{VKch}
\D^2 z^j(\cdot, 0) = M^j.
\end{equation}

Finally, for $\e$ small enough we define
$\t w^j_{\e} : S\to\R$ by $\t w^j_{\e}(x_1, x_2) = z^j(x_1, \e x_2)$.
From \eqref{VKz} it follows immediately that $\t w^j_{\e}\to w^j$ strongly in $L^2(S)$, as $\e\to0$. Moreover, since
$$
\D_{\e} \t w^j_{\e}(x)=\nabla z^j(x_1,\e x_2),
$$
equation \eqref{konstruktion-2} implies that 
$\D_{\e} \t w^j_{\e} \to ((w^j)', \vartheta^j)$ strongly in $W^{1,2}(S; \R^2)$, as $\e\to0$.
In particular, denoting by $F_{\e}\in\R^2$ the average of $\D_{\e} \t w^j_{\e}$ over $S$,
we have
\begin{align*}
\lim_{\e\to 0} F_{\e} 
= \frac{1}{\ell}\int_I 
\zwo{(w^j)'}{\theta^j}(x_1)\, dx_1 = 0,
\end{align*}
by definition of $w^j$ and $\theta^j$.
Similarly, 
denoting by $c_{\e}$ the average of $\t w^j_{\e}$ over $S$, we have $c_{\e}\to 0$.
Hence the functions $w^j_{\e} : S\to\R$ defined by
$$
w^j_{\e}(x) := \t w^j_{\e}(x) - F_{\e}\cdot\zwo{x_1}{\e x_2} - c_{\e},
$$
still satisfy $\t w^j_{\e}\to w^j$ strongly in $L^2(S)$ and 
$\D_{\e} \t w^j_{\e} \to ((w^j)', \vartheta^j)$ strongly in $W^{1,2}(S; \R^2)$.
Moreover, $w^j_{\e}\in W^{2,2}_{\langle 0\rangle}(S)$ by definition
of $F_{\e}$ and $c_{\e}$.

Finally, since
$$
\nabla^2_\e w_\e^j(x) =
\nabla^2_\e \t w_\e^j(x)=
 \nabla^2 z^j(x_1,\e x_2), 
$$
we have that $w_\e^j\in\We$. By \eqref{VKch} we deduce that $\nabla^2_\e w_\e^j\to M^j$ strongly in $L^{2}(S; \R^{2\times 2}_{\sym})$, as $\e\to0$.
Hence,
$$
\lim_{\e\to 0} J_\e^{CvK}(w_\e^j)=\lim_{\e\to0}\frac{1}{24}\int_S Q_2(\nabla^2_\e w_\e^j(x))\, dx= \frac{1}{24}\int_S Q_2(M^j(x))\, dx
= \frac{1}{24}\F(M^j) .
$$
Therefore, by taking diagonal sequences we obtain the desired maps.
\end{proof}

\bigskip

\noindent
{\bf Acknowledgements.}
MGM acknowledges support by GNAMPA--INdAM under Project 2016 ``Multiscale analysis of complex systems with variational methods''
and by the ERC under Grant No.\ 290888 ``Quasistatic and Dynamic Evolution Problems in Plasticity and Fracture''. PH acknowledges support by the DFG.

\end{document}